\documentclass[11pt]{amsart}
\usepackage{amscd}
\usepackage{amsmath}
\usepackage{amsxtra}
\usepackage{amsfonts}
\usepackage{amssymb}

\newtheorem{theorem}{Theorem}[section]
\newtheorem{corollary}[theorem]{Corollary}
\newtheorem{lemma}[theorem]{Lemma}
\newtheorem{proposition}[theorem]{Proposition}

\newtheorem{conjecture}[theorem]{Conjecture}
\theoremstyle{definition}
\newtheorem{definition}[theorem]{Definition}
\newtheorem{remark}[theorem]{Remark}

\theoremstyle{remark}

\renewcommand{\theclaim}{\textup{\theclaim}}

\newtheorem*{acknowledgements}{Acknowledgements}

\numberwithin{equation}{section}

\def\openone

{\mathchoice

{\hbox{\upshape \small1\kern-3.3pt\normalsize1}}

{\hbox{\upshape \small1\kern-3.3pt\normalsize1}}

{\hbox{\upshape \tiny1\kern-2.3pt\SMALL1}}

{\hbox{\upshape \Tiny1\kern-2pt\tiny1}}}

\makeatletter

\newbox\ipbox

\newcommand{\diracb}[1]{\left\langle #1\mathrel{\mathchoice

{\setbox\ipbox=\hbox{$\displaystyle \left\langle\mathstrut
#1\right.$}

\vrule height\ht\ipbox width0.25pt depth\dp\ipbox}

{\setbox\ipbox=\hbox{$\textstyle \left\langle\mathstrut
#1\right.$}

\vrule height\ht\ipbox width0.25pt depth\dp\ipbox}

{\setbox\ipbox=\hbox{$\scriptstyle \left\langle\mathstrut
#1\right.$}

\vrule height\ht\ipbox width0.25pt depth\dp\ipbox}

{\setbox\ipbox=\hbox{$\scriptscriptstyle \left\langle\mathstrut
#1\right.$}

\vrule height\ht\ipbox width0.25pt depth\dp\ipbox}

}\right. }

\newcommand{\dirack}[1]{\left. \mathrel{\mathchoice

{\setbox\ipbox=\hbox{$\displaystyle \left.\mathstrut
#1\right\rangle$}

\vrule height\ht\ipbox width0.25pt depth\dp\ipbox}

{\setbox\ipbox=\hbox{$\textstyle \left.\mathstrut
#1\right\rangle$}

\vrule height\ht\ipbox width0.25pt depth\dp\ipbox}

{\setbox\ipbox=\hbox{$\scriptstyle \left.\mathstrut
#1\right\rangle$}

\vrule height\ht\ipbox width0.25pt depth\dp\ipbox}

{\setbox\ipbox=\hbox{$\scriptscriptstyle \left.\mathstrut
#1\right\rangle$}

\vrule height\ht\ipbox width0.25pt depth\dp\ipbox}

} #1\right\rangle}

\newcommand{\beq}{\begin{equation}}

\newcommand{\eeq}{\end{equation}}

\newcommand{\bz}{\mathbb{Z}}

\newcommand{\br}{\mathbb{R}}

\newcommand{\bn}{\mathbb{N}}

\def\blfootnote{\xdef\@thefnmark{}\@footnotetext}

\newcommand{\Laba}{\textup{\L}}

\renewcommand{\mod}{\operatorname{mod}}

\input xy
\xyoption{all}
\usepackage{amssymb}

\def\T{\mathcal{T}}

\def\-{^{-1}}

\newcommand{\lcm}{\textup{lcm}}

\begin{document}

\title[Some reductions of the spectral set conjecture to integers]{Some reductions of the spectral set conjecture to integers}
\author{Dorin Ervin Dutkay}

\address{[Dorin Ervin Dutkay] University of Central Florida\\
	Department of Mathematics\\
	4000 Central Florida Blvd.\\
	P.O. Box 161364\\
	Orlando, FL 32816-1364\\
U.S.A.\\} \email{Dorin.Dutkay@ucf.edu}

\author{Chun-Kit Lai}

\address{[Chun-Kit Lai] McMaster University\\\\
	Department of Mathematics and Statistics\\
	120 Main Street West\\
	Hamilton, Ontario\\
	Canada L8S 4K1\\}

\email{cklai@math.mcmaster.ca}

\thanks{}
\subjclass[2000]{42A32,05B45}
\keywords{Fuglede conjecture, Coven-Meyerowitz conjecture, spectral, tile, Fourier basis}

\begin{abstract}
The spectral set conjecture, also known as the Fuglede conjecture, asserts that every bounded spectral set is a tile and vice versa.  While this conjecture remains open on ${\mathbb R}^1$, there are many results in the literature that discuss the relations among various forms of the Fuglede conjecture on ${\mathbb Z}_n$, ${\mathbb Z}$ and ${\mathbb R}^1$ and also the seemingly stronger universal tiling (spectrum) conjectures on the respective groups.  In this paper, we clarify the equivalences between these statements in dimension one. In addition, we show that if the Fuglede conjecture on ${\mathbb R}^1$ is true, then every spectral set with rational measure must have a rational spectrum. We then investigate the Coven-Meyerowitz property for finite sets of integers, introduced in \cite{CoMe99}, and we show that if the spectral sets and the tiles in ${\mathbb Z}$ satisfy the Coven-Meyerowitz property, then both sides of the Fuglede conjecture on ${\mathbb R}^1$ are true.
\end{abstract}
\maketitle \tableofcontents

\section{Introduction}

Let $\Omega$ be a bounded measurable set on ${\mathbb R}^d$ of positive finite Lebesgue measure. We say that it is a {\it spectral set} if there exists orthogonal basis of the form $\{e^{2\pi i \langle\lambda,\cdot\rangle}\}_{\lambda\in\Lambda}$  in $L^2(\Omega)$ and $\Lambda$ is called a {\it spectrum} of $\Omega$. In studying the extension of commuting self-adjoint partial differential operators to general domains, Fuglede \cite{Fug74} introduced the concept of spectral sets and proposed a conjecture concerning their geometric characterization.

\begin{conjecture}
[Fuglede, 1974] $\Omega$ is a spectral set if and only if $\Omega$ is a translational tile.
\end{conjecture}

Recall that $\Omega$ is a {\it translational tile} if there exists a {\it tiling set} ${\mathcal T}$ such that $\bigcup_{t\in{\mathcal T}}(\Omega+t) = {\mathbb R}^d$ and the Lebesgue measure of $(\Omega+ t)\cap (\Omega+t')$ is zero of all $t\neq t'$ in ${\mathcal J}$. We also say that a countable set $\Lambda$ is {\it periodic} if there exists a {\it period} $a>0$ such that $\Lambda+a=\Lambda$. It is easy to see that the set of all periods ${\mathcal P}$ of $\Lambda$ form a countable subgroup of ${\mathbb R}$ and we call the set of all periods the {\it period lattice}. The {\it minimal period}, denoted by $p_{\min}$, is the smallest positive period in the period lattice. We can deduce easily that ${\mathcal P} = p_{\min}{\mathbb Z} $.

\medskip

 It was shown by Fuglede \cite{Fug74} that the conjecture is true if $\Omega$ is a fundamental domain for a lattice. He also showed that circles and triangles are not spectral. Moreover, he gave some examples of spectral tiles that are not fundamental domains. Recently, Tao \cite{Tao04} gave a counterexample to disprove the conjecture on $d\geq 5$. It was eventually shown that the conjecture is false in both directions on $d\geq 3$ \cite{KM06,FaMaMo06,M05,KoMa06}. All these counterexamples  involve the study of the Fuglede conjecture on finite abelian groups and also on the integer lattice, or they involve some counterexamples for the seemingly stronger conjectures called the {\it universal spectrum conjecture (USC)} and {\it universal tiling conjecture (UTC)} introduced in \cite{LaWa97,PeWa01}.

\medskip

It is still not known whether these techniques can be used to produce counterexamples for the Fuglede conjecture in dimensions one or two. On the other hand, there are numerous positive results indicating that the conjecture on ${\mathbb R}^1$ might be true. One of the most promising results is that any tiling sets and any spectra on ${\mathbb R}^1$ must be periodic. These results are  due to Lagarias and Wang \cite{LaWa96} (for tiling sets) and also recently by Iosevich and Kolountzakis \cite{IoKo12} (for spectra).  Motivated by these results, we believe it is useful to clarify the equivalences between various forms of the Fuglede conjecture on the groups ${\mathbb Z}_n$, ${\mathbb Z}$ and ${\mathbb R}^1$ and the corresponding universal tiling and spectrum conjectures on ${\mathbb Z}_n$ and ${\mathbb Z}$.

\medskip

\begin{definition}
Let $G$ be a locally compact abelian group and $\widehat{G}$ its dual group. A set $\Lambda\subset{\widehat{G}}$ is called a {\it spectrum} of a set $T\subset G$ if the characters $\{\lambda\}_{\lambda\in\Lambda}$ form an orthonormal basis in $L^2(T)$. $T$ is called a {\it spectral set} of $G$. $T$ is called a tile if there exists a {\it tiling set} ${\mathcal T}$ in $G$ such that $T\oplus {\mathcal T} = G$ (i.e. every element in $G$ can be uniquely written as sum of elements in $T$ and ${\mathcal T}$, up to Haar measure zero ).
\end{definition}

 Note that $\widehat{{\mathbb R}} = {\mathbb R}$, $\widehat{{\mathbb Z}} = {\mathbb T}$ and  ${\widehat{{\mathbb Z}_n}} = \frac1n{\mathbb Z}_n$. By a natural identification, we think of  ${\widehat{\bz_n}} $ as ${\mathbb Z}_n$ also. We consider the following statements for $G ={\mathbb R}, {\mathbb Z}, {\mathbb Z}_n$.

\vspace{0.2cm}

{\bf (S-T($G$))}. Every bounded spectral set in $G$ is a tile in $G$.

\vspace{0.2cm}

{\bf (T-S($G$))}. Every bounded tile in $G$ is a spectral set in $G$.

\vspace{0.2cm}

{\bf (USC($G$))}. If $\mathcal B$ is a finite family of bounded sets in $G$ with the property that there exists $A\subset G$ such that $A\oplus B=G$ for all $B\in\mathcal B$, then the sets $B\in\mathcal B$ have a common spectrum $\Gamma \subset\widehat{G}$.

\vspace{0.2cm}

{\bf (UTC($G$))}. If $\mathcal A$ is a finite family of bounded sets in $G$ with the property that there exists a common spectrum $\Gamma\subset{\widehat{G}}$,  then the sets $A\in\mathcal A$ have a common tiling set ${\mathcal T}$ so that $A\oplus {\mathcal T} = {\mathbb Z}$.

\medskip

Unless otherwise mentioned, we will write {\bf (S-T(${\mathbb Z}_n$))} to mean that spectral implies tile is true on ${\mathbb Z}_n$ {\it for all $n\in{\mathbb N}$}. The same applies for the other statements.  This helps simplify the statements. On ${\mathbb Z}$, we also consider the following stronger statements.

\vspace{0.2cm}

 {\bf (Strong S-T(${\mathbb Z}$))}. For every finite union of intervals $\Omega = A+[0,1]$ with $A\subset {\mathbb Z}$,  if $\Lambda$ is a spectrum of $\Omega$ with minimal period $\frac{1}{N}$, then $\Omega$ tiles ${\mathbb R}$ with a tiling set ${\mathcal T}\subset N{\mathbb Z}$.

\vspace{0.2cm}

{\bf (Strong T-S(${\mathbb Z}$)) }. For every finite union of intervals $\Omega = A+[0,1]$ with $A\subset {\mathbb Z}$,  if $\Omega$ tiles ${\mathbb R}$ with a tiling set ${\mathcal T}\subset k{\mathbb Z}$ for some integer $k>0$, then $\Omega$ admits a spectrum $\Lambda$ with period $\frac1k$.

\medskip

We will consider the equivalence of the statements above with $G = {\mathbb R},{\mathbb Z}$ and ${\mathbb Z}_n$. Results of this type are scattered in the literature \cite{LaWa97,PeWa01,LS01,M05,FaMaMo06,KoMa06,DuJo12b}. We gather them and some new results in the following theorem to describe the complete picture.

\begin{theorem}\label{th0.0}

\vspace{0.2cm}

(i)   {\bf  (T-S(${\mathbb Z}_n$))} $\Longleftrightarrow$ {\bf (USC(${\mathbb Z}_n$))}  $\Longleftrightarrow$  {\bf (T-S(${\mathbb R}$))} $\Longleftrightarrow$  {\bf (T-S(${\mathbb Z}$))}.

\vspace{0.2cm}

(ii)   {\bf (S-T(${\mathbb Z}_n$))}  $\Longleftrightarrow$ {\bf (UTC(${\mathbb Z}_n$))},

\vspace{0.1cm}

 \ \ \ \ \  {\bf (S-T(${\mathbb R}$))} $\Longrightarrow${\bf (S-T(${\mathbb Z}$))} $\Longrightarrow$ {\bf (S-T(${\mathbb Z}_n$))}. If  any spectral set $\Omega$ of Lebesgue measure 1 has a rational spectrum, then the converses hold.

\vspace{0.2cm}

(iii)  {\bf (Strong T-S(${\mathbb Z}$))} $\Longleftrightarrow$ {\bf (USC(${\mathbb Z}$))} $\Longleftrightarrow$  {\bf (T-S(${\mathbb R}$))}.

\vspace{0.2cm}

(iv)  {\bf (Strong S-T(${\mathbb Z}$))} $\Longleftrightarrow$ {\bf (UTC(${\mathbb Z}$))} $\Longleftrightarrow$  {\bf (S-T(${\mathbb R}$))}.
\end{theorem}

\begin{remark}
  It is not known whether the statements {\bf (S-T(${\mathbb R}$))} is also equivalent to  {\bf (S-T(${\mathbb Z}$))}. For the statements of {\bf (Strong T-S(${\mathbb Z}$))} or {\bf (Strong S-T(${\mathbb Z}$))}, one should be careful with the additional requirements on the tiling sets and on the period of the spectra. From Proposition \ref{prop2.4}, the statement for $k=1$ is always true. On the other hand, if we can find an integer tile $A+[0,1)$ such that it tiles with a subset of $k{\mathbb Z}$ with $k>1$, but there is no spectrum of the correct period, then Fuglede conjecture will be disproved.
\end{remark}

As we see  in Theorem \ref{th0.0}(ii), the rationality of the spectra is the property that is breaking the symmetry between the statements for the two sides of the duality: we do not know if all spectral sets of rational measure must have a rational spectrum. However, we will prove that rationality of the spectrum is a consequence of the Fuglede conjecture.

\begin{theorem}
Suppose the Fuglede conjecture is true (i.e. both {\bf (S-T(${\mathbb R}$))} and {\bf (T-S(${\mathbb R}$))} holds), then any bounded spectral sets $\Omega$ of $|\Omega|=1$ will have a rational spectrum.
\end{theorem}

In fact, using the same reasoning, we can also show that the Fuglede conjecture implies every bounded tile of measure 1 will have a rational tiling set, which is known to be true in \cite{LaWa96}.


\medskip

The above equivalences infer us that the Fuglede problem is related to the factorization of the abelian groups ${\mathbb Z}_n$. Lagarias and Wang \cite{LaWa97} showed that the {\it Tijdeman conjecture} implies {\bf (T-S(${\mathbb R}$))}. However, the Tijdeman conjecture was proved to be false (see \cite{CoMe99} and \cite{LS01}) using a construction of integer tiles by \cite{S85}. At the same time, Coven and Meyerowitz introduced two algebraic properties on the finite sets $A\subset \bz^{+}\cup\{0\}$. Define {\it the mask polynomial} associated to $A$,
$$
A(x) := \sum_{a\in A}x^a.
$$
Recall that the cyclotomic polynomial $\Phi_s(x)$ is the minimal polynomial for the primitive $s^{th}$ root of unity.


\medskip

\begin{definition}\label{def0.1}
Let $A$ be a finite subset of $\bz^{+}\cup\{0\}$ and let
$${\mathcal S}_A = \{p^{\alpha}:  p\ \mbox{is a prime}, \alpha\geq 1 \ \mbox{an integer} \ \mbox{and} \  \Phi_s(x) \  \mbox{divides} \ A(x) \}.$$
We say that $A$ (or $A(x)$) satisfies the Coven-Meyerowitz property (CM-property) if $A(x)$ satisfies

\vspace{0.2cm}

$(T1)$. $\#A = A(1) = \prod_{s\in {\mathcal S}_A}\Phi_s(1)$.

\vspace{0.2cm}

$(T2)$. If $s_1,\cdots, s_n\in {\mathcal S}_A$, then $\Phi_{s_1\cdots s_n}(x)$ divides $A(x)$.
\end{definition}

\medskip

  Coven-Meyerowitz showed that tiles on ${\mathbb Z}$ must satisfy (T1) and they satisfy (T2) if the number of elements in the tiles contains at most 2 prime factors. They conjectured that all integer tiles must satisfy the CM-property. It is now known that Coven-Meyerowitz conjecture is strictly weaker than the Tijdeman conjecture \cite{LS01}. Moreover, the work of Coven and Meyerowitz \cite{CoMe99} and {\Laba}aba \cite{Lab02} tell us that:

\begin{theorem} [Coven-Meyerowitz, {\Laba}aba]

(i)\cite{CoMe99} If $A\subset{\mathbb Z}^{+}\cup\{0\}$ satisfies the CM-property, then $A$ is a tile of integers and $A+[0,1]$ is a tile of ${\mathbb R}$ with tiling set ${\mathcal T}_{CM}$.

\vspace{0.2cm}

 (ii)\cite{Lab02} If $A\subset{\mathbb Z}^{+}\cup\{0\}$ satisfies the CM-property, then $A$ is a spectral set of integers, and $A+[0,1]$ is a spectral set of ${\mathbb R}$ with spectrum $\Lambda_{{\Laba}}$.
\end{theorem}

Due to the explicit construction of the tiling set ${\mathcal T}_{CM}$ and of the spectrum $\Lambda_{{\Laba}}$ under the CM assumption (see Definition \ref{def1.1} and \ref{def1.2}), using the above theorems, we can show that the CM-property implies both sides of the Fuglede conjecture.

\medskip

\begin{theorem}\label{th0.2}

(i) If every spectral set $A\subset\bz^{+}\cup \{0\}$ satisfies the CM-property, then every bounded spectral subset of $\br$ tiles by translations.

\vspace{0.2cm}

(ii) If every tile $A\subset\bz^{+}\cup \{0\}$ satisfies the CM-property, then every bounded tile  of $\br$ is a spectral set.
\end{theorem}

 Although the second part of the theorem is probably known to some authors, we have not found a proof in the literature so we provide here a complete one.

\medskip

Having these results, we propose the following questions which seem to be essential for solving the Fuglede conjecture on ${\mathbb R}$.

{\bf (Q1).} Is it true that every spectral set of  Lebesgue measure 1 has some rational spectrum?

\medskip

{\bf (Q2).} Do spectral sets and tiles on ${\mathbb Z}$ satisfy the Coven-Meyerowitz property?

\medskip

We strongly believe that the first question has a positive answer. For the second question, the tile part is exactly the Coven-Meyerowitz conjecture. On the spectral side, one needs to answer (Q1) in order to give some positive answers, since irrational spectra can never give any cyclotomic polynomial factors for $A(x)$.

\medskip

 We organize our paper as follows. In Section 2, we collect some basic known results about the geometric structure of tiles and spectral sets on ${\mathbb R}^1$. In Section 3, we study in detail the equivalences between the mentioned statements and prove Theorem \ref{th0.0} in Theorem \ref{th5.2}, Proposition \ref{prop3.1} and Theorem \ref{th1.6}. In Section 4, we discuss the Coven-Meyerowitz property and prove Theorem \ref{th0.2}.

\medskip

\medskip
\section{Preliminaries}
In this section, we collect some of the basic known results concerning the geometric structure of tiles and spectral sets that will be used throughout the paper.
First, one dimensional translational tiles were completely characterized by Lagarias and Wang \cite{LaWa96}, we summarize this as follows.

\begin{theorem}\label{th2.1}
Let $T$ be a bounded measurable translational tile on ${\mathbb R}$ with tiling set ${\mathcal T}$. Then the tiling set ${\mathcal T}$ must be periodic with period an integral multiple of $|T|$ and ${\mathcal T}\subset |T| \cdot{\mathbb Q}$.

Moreover, if the period lattice contains ${\mathbb Z}$, then we have the following decomposition of $T$ and ${\mathcal T}$: there exists a positive integer $L$ and a finite collection of finite subsets of integers ${\mathcal B}$   such that
$$
T = \bigcup_{B\in{\mathcal B}}(T_B+\frac{1}{L} B); \ {\mathcal T} = \frac{A}{L}+{\mathbb Z}
$$
where (i) $A\oplus B = {\mathbb Z}_L$ for all $B\in{\mathcal B}$, (ii) $\bigcup_{B\in{\mathcal B}} T_B = [0,1/L)$ with disjoint union.
\end{theorem}

\begin{remark}
Lagarias and Wang proved this theorem under the assumption that $T$ is a bounded region (i.e. $|\partial T|=0$). However,  \cite[Theorem 6.1]{KL96} showed that any tiling for measurable sets must be periodic. The proof of the rationality is Fourier-analytic and it does not require the region assumption. For the decomposition, the region condition is also not needed as long as we do not require $T_B$ to be a region. Hence, the above theorem is true for any measurable sets.
\end{remark}

\medskip

There is no complete classification of spectral sets like Theorem \ref{th2.1}. One reason is that there is no rationality result on the spectra. However, the recent result of \cite{IoKo12} (see also \cite{DuJo12b}) made a step forward toward the classification of spectral sets. We say that $\Omega$ is a {\it $p$-tile by $\frac{1}{p}{\mathbb Z}$} if
$$
\sum_{k\in{\mathbb Z}}\chi_{\Omega}(x+\frac{k}{p}) = p. \ \mbox{a.e.}.
$$
In other words, $\Omega$ covers the real line $p$ times when it is translated by $\frac{1}{p}{\mathbb Z}$.
\medskip

\begin{theorem}\label{th2.2}
Any spectra $\Lambda$ of the spectral set $\Omega$ on ${\mathbb R}^1$ must be periodic with period an integral multiple of $|\Omega|^{-1}$. Moreover, $\Omega$ must be a $p$-tile by $\frac{1}{p}{\mathbb Z}$ where $p$ is the period of $\Lambda$.
\end{theorem}

\medskip

From Theorem \ref{th2.2}, since $\Omega$ is a $p$-tile by $\frac{1}{p}{\mathbb Z}$, we have the following natural decomposition of a spectral set $\Omega$ on ${\mathbb R}^1$. For $x\in\frac{1}{p}{\mathbb Z}$ and $S\subset {\mathbb Z}$ with $\#S=p$, where $p$ is the period,
$$
\Omega_x = \{k\in{\mathbb Z}: x+\frac{k}{p}\in\Omega\}, \  A_S = \{x\in[0,\frac{1}{p}): \Omega_x = S\}
$$
\begin{equation}\label{eq2.1}
\Omega = \bigcup_{|S|=p}\left(A_S+\frac{S}{p}\right), \ \bigcup_{|S|=p}A_S = [0,\frac{1}{p}).
\end{equation}
The following proposition can be found in \cite[Proposition 2.8 and 2.9]{DuJo12b}. See also \cite{Ped96}. It describes the relation between the spectral property of $p$-tiles and their decomposition.

\begin{proposition}\label{prop2.3}
Let $\Omega$ be a $p$-tile by $\frac{1}{p}{\mathbb Z}$. Then a set $\Lambda$ of the form $\Lambda = \{0,\lambda_1,\cdots,\lambda_{p-1}\}+p{\mathbb Z}$ is a spectrum of $\Omega$ if and only if $\{0,\lambda_1\cdots,\lambda_{p-1}\}$ is a spectrum for $\frac{1}{p}\Omega_x$ for a.e. $x\in[0,\frac{1}{p})$. Moreover, if $\Omega_x$ have a common tiling set ${\mathcal T}$, then $\Omega$ tiles ${\mathbb R}^1$ by the tiling set $\frac{1}{p}{\mathcal T}$.
\end{proposition}

\medskip

We also need another simple proposition, which can be found, for the most part, in \cite{KoMa06}.
\begin{proposition}\label{prop2.4}
Let $\Lambda$ be a spectrum for $\Omega=A+[0,1)$, $0\in\Lambda$. Then $1$ is a period of $\Lambda$. Moreover, if we write $\Lambda = \Gamma+{\mathbb Z}$ with $\Gamma\subset[0,1)$, then $\Gamma$ is a spectrum for $A$ if and only if $\Lambda$ is a spectrum for $A+[0,1)$.
\end{proposition}

\begin{proof}
 Suppose $1$ is not a period, we can find $\lambda$ such that $\lambda+1\not\in\Lambda$. For any $\lambda'\in\Lambda$,
$$
\int_{\Omega}e^{2\pi i (\lambda-\lambda'+1)x}dx = \frac{1}{2\pi i (\lambda-\lambda'+1)}\sum_{a\in A} (e^{2\pi i (\lambda-\lambda'+1)(a+1)}-e^{2\pi i (\lambda-\lambda'+1)a}) $$$$= \frac{1}{2\pi i (\lambda-\lambda'+1)}\sum_{a\in A} (e^{2\pi i (\lambda-\lambda')(a+1)}-e^{2\pi i (\lambda-\lambda')a})=\frac{\lambda-\lambda'}{\lambda-\lambda'+1}\int_{\Omega}e^{2\pi i (\lambda-\lambda')x}dx.
$$
It is clear that if $\lambda=\lambda'$, the quantity above is zero. If $\lambda\neq\lambda'$, the orthogonality of $\Lambda$ shows that the last integral is zero. Therefore, $e^{2\pi i (\lambda+1)x}$ is orthogonal to all elements in  $\Lambda$ which contradicts to the completeness of the spectrum. Hence, 1 must be a period.  The last statement is well known (see e.g. \cite{KoMa06}).
\end{proof}

\medskip

Another important result we will need is the decomposition theorem of the tile if the tiling set is a subset of $k{\mathbb Z}$, due to Coven and Meyerowitz \cite[Lemma 2.5]{CoMe99}.

\begin{theorem}\label{th2.3}
Let $A$ be a finite subset of non-negative integers such that  $A\oplus C = {\mathbb Z}$ for some $C$. Suppose that $C\subset k{\mathbb Z}$ for some $k>1$.  If we define $a_i = \min\{a\in A: a\equiv i (\mbox{{\rm mod}} \ k)\}$ for $i=0,\cdots,k-1$ and
$$
A_i = \{\frac{a-a_i}{k}: a\equiv i (\mbox{{\rm mod}} \ k)\},
$$
then
\vspace{0.2cm}
(i) $A(x) = x^{a_0}A_0(x^{k})+\cdots x^{a_{k-1}}A_{k-1}(x^{k})$.

\vspace{0.2cm}

(ii) $A_i\oplus \frac{1}{k}C = {\mathbb Z}$.

\vspace{0.2cm}

(iii)  $A$ is equidistributed ({\rm mod} $k$) (i.e. $\#A_i = \frac{\#A}{k}$ for all $i$), In particular, $k$ divides $\#A$.
\end{theorem}

\medskip

\medskip

\section{Fuglede Conjectures on ${\mathbb Z}_n$, ${\mathbb Z}$ and ${\mathbb R}$ }

In this section, we will discuss in detail the relations between the Fuglede problems in ${\mathbb Z}_n$, ${\mathbb Z}$ and ${\mathbb R}$ and  Theorem \ref{th0.0} will be proved. First, it was proved in \cite[Proposition 1, Remark 2]{FaMaMo06} that

\medskip

\begin{center}
{\bf (T-S(${\mathbb Z}_n$))}  $\Longleftrightarrow$ {\bf (USC(${\mathbb Z}_n$))} and {\bf (S-T(${\mathbb Z}_n$))}  $\Longleftrightarrow$ {\bf (UTC(${\mathbb Z}_n$))}.
\end{center}
 (Recall that the statement {\bf (T-S(${\mathbb Z}_n$))} means {\bf (T-S(${\mathbb Z}_n$))} is true for all $n\in{\mathbb N}$.)
\medskip

To make the connection to the statements on ${\mathbb R}^1$, we need the following theorem, which has been essential in constructing counterexamples for the Fuglede conjecture in high dimensions \cite{Tao04,KoMa06}. One may refer to \cite[Proposition 2.1 and 2.5]{M05} and \cite[Theorem 4.1]{KoMa06} for the proofs.

\begin{theorem}\label{th5.1}
 Let $A\subset {\mathbb Z}_n$ and let $T(k) = n\{0,1,\cdots,k-1\}$. Define $B(k) = A+T(k)$. Then for large enough $k$,

 (i) $B(k)+[0,1)$ is a spectral set on ${\mathbb R}^1$ if and only if $A$ is spectral in ${\mathbb Z}_n$.

 \vspace{0.2cm}

 (ii) $B(k)+[0,1)$ is a tile on ${\mathbb R}^1$ if and only if $A$ is tile in ${\mathbb Z}_n$
\end{theorem}

The following is our main conclusion.

\begin{theorem}\label{th5.2}
(i) {\bf (T-S(${\mathbb Z}_n$))} $\Longleftrightarrow$ {\bf (USC(${\mathbb Z}_n$))} $\Longleftrightarrow$ {\bf (T-S(${\mathbb R}$))} $\Longleftrightarrow$  {\bf (T-S(${\mathbb Z}$))}.

\vspace{0.2cm}

(ii) {\bf (S-T(${\mathbb Z}_n$))} $\Longleftrightarrow$ {\bf (UTC(${\mathbb Z_n}$))} and {\bf (S-T(${\mathbb R}$))}  $\Longrightarrow$ {\bf (S-T(${\mathbb Z}$))}  $\Longrightarrow$ {\bf (S-T(${\mathbb Z}_n$))}. If  any spectral set $\Omega$ of Lebesgue measure 1 has a rational spectrum, then the converse holds.
\end{theorem}

\begin{proof}
(i) The first equivalence has been mentioned. We first show {\bf (USC(${\mathbb Z}_n$))}  $\Rightarrow$ {\bf (T-S(${\mathbb R}$))}. It suffices to show that a tile of measure 1 is a spectral set. Let $\Omega$ be a bounded measurable tile $|\Omega|=1$, with tiling set $\mathcal T$. By Theorem \ref{th2.1}, the tiling set is periodic, with period lattice $p\bz$ for some $p\in\bn$. We rescale $\Omega$ so that $\frac1p\Omega$ tiles with $\frac1p\mathcal T$ and period lattice $\bz$. By Theorem \ref{th2.1} we have that $\frac1p\Omega$ and $\frac1p\mathcal T$ have the following form
$$\frac1p\Omega=\bigcup_{B\in\mathcal B}(T_B+\frac1L B),\quad \bigcup_{B\in\mathcal B}T_B=[0,\frac 1L)\quad \mbox{(disjoint union)},$$
$$\frac1p\mathcal T=\frac1L A+\bz,\quad A\oplus B=\bz_L\mbox{ for all }B\in\mathcal B.$$
Since $|\Omega|=1$, we have that $\#A=p$ and therefore $p$ divides $L$, $\#B=\frac Lp=:k\in\bz$.

We have
$$\Omega=\bigcup_{B\in\mathcal B}\left(pT_B+\frac{p}{L}B\right)=\bigcup_{B\in\mathcal B}\left(pT_B+\frac{1}{k}B\right),\quad \bigcup_{B\in\mathcal B}pT_B=[0,\frac pL)=[0,\frac1k),\quad \mathcal T=\frac1k A+p\bz.$$

We now apply the {\bf (USC(${\mathbb Z}_n$))} on $B\in{\mathcal B}$ and $n=L$ (as the sets $B$ have a common tiling set $A$ in ${\mathbb Z}_L$), the sets $B$ with  $B\in\mathcal B$ have a common spectrum $\frac 1L\Gamma$ with $\Gamma\subset{\mathbb Z}_L$. Hence, $\frac{1}{k} B$ have a common spectrum $\frac{1}{kL}\Gamma$.

We now claim that the set $\Omega$ is a $k$-tile with $\frac1k\bz$. Then {\bf (T-S(${\mathbb R}$))} follows from Proposition \ref{prop2.3} and $\Omega$ has a spectrum $\frac{1}{kL}\Gamma+k{\mathbb Z}$.  To justify the claim, we first note that for any $x\in\br$ there exists  unique $y\in[0,\frac1k)$, $m\in\bz$ such that $x=y+\frac mk$. Then there is a unique $B\in\mathcal B$ such that $y\in pT_B$. Then $y+pb\in\Omega$ for all $b\in B$, so $x=(y+pb)+(\frac mk-pb)$. Thus $x$ belongs to at least $|B|=k$ translates of $\Omega$ from $\frac1k\bz$. If $x=z+\frac rk$ with $z\in\Omega$ and $r\in\bz$ then $z=y+pb$ for some $B\in\mathcal B$, $y\in pT_B$ and $y$ is uniquely determined in $[0,\frac1k)$ so $B$ is uniquely determined in $\mathcal B$. So $x=(y+pb)+\frac {m'}k$ as in the decomposition above. Thus $\Omega$ $k$-tiles by $\frac1k\bz$.

\medskip

 {\bf (T-S($\br$))} implies {\bf (T-S($\bz$))} with Proposition \ref{prop2.4}. Assume now {\bf (T-S(${\mathbb Z}$))} and take $A$ a tile in ${\mathbb Z}_n$. Assume on the contrary $A$ is not spectral in ${\mathbb Z}_n$, then Theorem \ref{th5.1}(i) implies that we can find some $k>1$ such that $A+n\{0,\cdots,k-1\}+[0,1)$ is not spectral on ${\mathbb R}^1$, so $A+n\{0,\cdots,k-1\}$ is not spectral in $\bz$ (by Proposition \ref{prop2.4}). However, $A+n\{0,\cdots,k-1\}$ is still a tile on ${\mathbb Z}$ by Theorem \ref{th5.1}(ii). This contradicts {\bf (T-S(${\mathbb Z}$))} and this shows {\bf (T-S(${\mathbb Z}_n$))}. This completes the proof of (i).

\medskip

(ii) The first equivalence was mentioned above. {\bf (S-T($\br$))} implies {\bf (S-T($\bz$))} with Proposition \ref{prop2.4}. Assume {\bf (S-T(${\mathbb Z}$))} and take a spectral set $A$ in ${\mathbb Z}_n$. We can prove in the same way as in (i) using Theorem \ref{th5.1} that {\bf (S-T(${\mathbb Z}_n$))} has to be true.
Finally, assume that every bounded spectral set of measure 1 has a rational spectrum. It suffices to show {\bf (S-T(${\mathbb R}$))} for $|\Omega|=1$. Take $\Omega$ spectral on ${\mathbb R}$ with spectrum $\Gamma+pZ$, by rationality,
we may write $\Gamma = \frac1L A$, with $A$ is a finite subset of integers and $L$
 a positive integer. Then Proposition \ref{prop2.3} implies $\frac1L A$ is a spectrum for a.e. $1/p\Omega_x$. Hence,
 $A$ has a spectrum $\frac{1}{pL} \Omega_x$ in $\bz_{pL}$. By the {\bf (UTC(${\mathbb Z}_n$))}, the sets $\Omega_x$ have a common tiling set in $\bz_{pL}$ for almost all $x$ and
 therefore $\Omega_x$ have a common tiling set in ${\mathbb Z}$. By the last statement in Proposition \ref{prop2.3}, $\Omega$ has a tiling set.
\end{proof}

\medskip

 We now  prove Theorem \ref{th0.0}(iii) and (iv). (iv) can be found in \cite{DuJo12b} except that the universal tiling conjecture in formulated in the following form.

\vspace{0.2cm}

{\bf (UTC($p$))}. Let $p\in{\mathbb N}$. Let $\Gamma$ be a subset of ${\mathbb R}$ which has $p$ elements and let $K>0$. If $\Gamma$ has a spectrum of the form $\frac1p A$ with $A\subset{\mathbb Z}$, $\max|A|\leq K$. Then there exists a tiling set ${\mathcal T}\subset{\mathbb Z}$ such that any spectrum of $\Gamma$ of the form $\frac1p A'$ with $A'\subset{\mathbb Z}$, $\max|A'|\leq K$ satisfies $A'\oplus {\mathcal T} = {\mathbb Z}$.

\vspace{0.2cm}

In \cite{DuJo12b} it is proved that {\bf (Strong (S-T(${\mathbb Z}$))} $\Longleftrightarrow$ {\bf (UTC($p$))} for all $p$ $\Longleftrightarrow$ {\bf (S-T(${\mathbb R}$))}. Therefore Theorem \ref{th0.0}(iv) will follow if we establish the following proposition.

\begin{proposition}\label{prop3.1}
{\bf (UTC($p$))} for all $p$ is equivalent to {\bf (UTC(${\mathbb Z}$))}.
\end{proposition}

\begin{proof}
Assume {\bf (UTC($p$))} holds. Let ${\mathcal A}$ be a finite collections of sets in ${\mathbb Z}$ with common spectrum $\Gamma$. Let $p = \#\Gamma$, then $p\Gamma$ is a spectral set and $\frac 1p A$ are spectra of $\Gamma$ for all $A\in{\mathcal A}$. By the {\bf (UTC($p$))}, the sets $A$ have a common tiling set ${\mathcal T}$ in ${\mathbb Z}$. This shows {\bf (UTC(${\mathbb Z}$))}.

\medskip

Conversely, let $p$, $\Gamma$ and $K$ be given. We consider ${\mathcal A} = \{A\subset {\mathbb Z}: \max|A|\leq K, \frac{1}{p} A \ \mbox{is a spectrum of} \ \Gamma \}$. It is clearly a finite collection of sets. Moreover, all $A\in{\mathcal A}$ have a common spectrum $\frac1p \Gamma$. Therefore, by {\bf (UTC(${\mathbb Z}$))}, the sets $A$ have a common tiling set ${\mathcal T}$ in ${\mathbb Z}$.
\end{proof}

\medskip

We now prove Theorem \ref{th0.0}(iv). As in \cite{DuJo12b}, we add one more equivalent statements.

\begin{theorem}\label{th1.6}
The following statements are equivalent
\begin{enumerate}
	\item {\bf (Strong T-S($\mathbb Z$)).} For every set $\Omega$ of the form $\Omega=A+[0,1)$ with $A\subset\bz$, $A$ finite, if $\Omega$ tiles $\br$ with a tiling set $\mathcal T\subset k\bz$, $k\in\bn$, then it has a spectrum of period $\frac1k$.
	\item For every set $\Omega$ of the form $\Omega=\cup_{i=1}^n(\alpha_i,\beta_i)$ with $\alpha_i,\beta_i$ rational and $|\Omega|=1$, if $\Omega$ tiles $\br$ with a tiling set $\mathcal T\subset \frac1p\bz$, $p\in\bn$, then it has a spectrum of period $p$.
	\item {\bf (USC(${\mathbb Z}$))}. If $\mathcal B$ is a finite family of sets in $\bz$ with the property that there exists $A\subset \bz$ such that $A\oplus B=\bz$ for all $B\in\mathcal B$, then the sets $B\in\mathcal B$ have a common spectrum $\Gamma \subset[0,1)$.
	\item {\bf (T-S(${\mathbb R}$))}. Every bounded measurable tile $\Omega$ in $\br$ is spectral.
\end{enumerate}
\end{theorem}

\begin{proof}
We first prove (i)$\Longrightarrow$(ii)$\Longrightarrow$(iii)$\Longrightarrow$(i) and then (iii)$\Longleftrightarrow$(iv).


\medskip

(i)$\Rightarrow$(ii). Take $\Omega,p$ and $\mathcal T$ as in (ii). We can assume $0=\alpha_1<\beta_1<\alpha_2<\dots<\alpha_n<\beta_n$. Let $N$ be a common denominator for all the rational numbers $\alpha_i,\beta_i$ which is also a multiple of $p$. Then $N\Omega$ is of the form $A+[0,1)$ with $A\subset\bz$ and $N\Omega$ tiles by $N\mathcal T$ which is a subset of $\frac{N}p\bz$. By the hypothesis, $N\Omega$ has a spectrum of period $\frac pN$. Therefore $\Omega$ has a spectrum of period $p$.

%

\medskip

(ii)$\Rightarrow$(iii). Let $\mathcal B=\{B_1,\dots,B_n\}$, $|B|=p$, $|A|=L/p=:k$. Take a rational partition $0=r_0<r_1<\dots<r_n=\frac1p$. Define
$$\Omega:=\bigcup_{i=1}^n\left((r_i,r_{i+1})+\frac1pB_i\right).$$
Then $\Omega$ tiles $\br$ with $\frac1p(A\oplus L\bz)$ because $\frac1p(B_i\oplus A\oplus L\bz)=\frac1p\bz$ and $\cup_i(r_i,r_{i+1})=[0,\frac1p)$ (of course, we ignore measure zero sets).
The hypothesis implies that $\Omega$ has a spectrum $\Lambda$ of period $p$ and we can assume $0\in\Lambda$. Then $\Lambda=\Gamma+p\bz$ where $\Gamma=\Lambda\cap[0,p)$. Note also that $\Omega$ is a $p$-tile by $\frac1p\bz$ and the sets $B_i$ appear as the sets $\Omega_x$ in Proposition \ref{prop2.3}. It  follows from  Proposition \ref{prop2.3} that $\Gamma$ is a spectrum for all the sets $\frac1p B_i$. Hence, the sets $B_i$ have a common spectrum $\frac1p\Gamma$.

\medskip

(iii)$\Rightarrow$(i). Let $\Omega = A+[0,1)$ as in (i). We may assume all elements in $A$ are non-negative. If $k=1$, then we consider the finite collection ${\mathcal B} = \{A\}$ as in (iii). Since tilings on ${\mathbb R}^1$ are periodic, we can find $A'$ and $L$ such that $A'\oplus A = {\mathbb Z}_L$. The hypothesis implies that $A$ is spectral with some spectrum $\Gamma$. Hence, $\Omega$ is spectral with spectrum $\Gamma +{\mathbb Z}$. This is clearly of period 1.

If now $k>1$, then $A$ has a tiling set ${\mathcal T}\subset k{\mathbb Z}$. By Theorem \ref{th2.3}(i), $A(x)$ can be written as
\begin{equation}\label{eq3.1}
A(x) = x^{a_0}A_0(x^{k})+\cdots x^{a_{k-1}}A_{k-1}(x^{k}).
\end{equation}
Consider the finite collections of sets ${\mathcal B} = \{A_0,\cdots,A_{k-1}\}$,  Theorem \ref{th2.3}(ii) implies that $A_i$ tiles with a common tiling set $\frac1k{\mathcal T}$ and hence the hypothesis tells us that there exists a common spectrum $\Gamma'\subset[0,1)$ for $A_i$. Let $\Gamma = \frac1k\Gamma'\oplus\{0,\frac{1}{k},\cdots \frac{k-1}{k}\}$ and we claim that $\Lambda = \Gamma+{\mathbb Z}$ is a spectrum of period $\frac1k$ of  $A+[0,1)$. First, it is clear that it has period $\frac1k$ since $\frac1k \gamma+\frac{j}{k} +\frac1k= \frac1k \gamma+\frac{j+1}{k}\in \Lambda$. To prove it is a spectrum,  it suffices to show $\Gamma$ is a spectrum of $A$ by Proposition \ref{prop2.4}. Note that elements of $\frac1k \Gamma'\subset[0,\frac1k)$, so elements in $\Gamma$ are all distinct and $\#\Gamma = (\#A_i)k = \#A$. Take distinct elements $y_1 := \frac1k \gamma+\frac{j}{k},y_2:=\frac1k \gamma'+\frac{j'}{k}$ in $\Gamma$, if $\gamma\neq\gamma'$, then
$$
A_i(e^{2\pi i (y_1-y_2)k}) = A_i(e^{2\pi i (\gamma-\gamma')}) =0
$$
because $\Gamma'$ is a spectrum for all $A_i$. (\ref{eq3.1}) implies $A(e^{2\pi i (y_1-y_2)})=0$. If $\gamma=\gamma'$, then $j\neq j'$ and then $y_1-y_2 = \frac{j-j'}{k}$, and with (\ref{eq3.1}), since all $\#A_i$ are equal (by Theorem \ref{th2.3}(iii)) and $a_i$ are complete residue (mod k), we have that
$$
A(e^{2\pi i (y_1-y_2)}) = \#(A_0)\left(e^{2\pi i (\frac{j-j'}{k})a_0}+\cdots+e^{2\pi i (\frac{j-j'}{k})a_{k-1}}\right) =0.
$$
This shows $\Gamma$ is a spectrum of $A$.

\medskip

(iii)$\Rightarrow$(iv). The proof is identical to the proof of {\bf (USC(${\mathbb Z}_n$))}  $\Rightarrow$ {\bf (T-S(${\mathbb R}$))} in  Theorem \ref{th5.2}.

\medskip

(iv)$\Rightarrow$(iii). By Theorem \ref{th5.2}, we know that {\bf (USC(${\mathbb Z}_n$))} holds. Let ${\mathcal B}$ be a finite collection of sets such that  there exists $A$ so that $A\oplus B = {\mathbb Z}$ for all $B\in{\mathcal B}$. Since tiling sets must be periodic, so $A = A'\oplus L{\mathbb Z}$ and hence $A'\oplus B = {\mathbb Z}_L$. From {\bf (USC(${\mathbb Z}_n$))}, there exists a common spectrum of $B$ in ${\frac{1}{n}{\mathbb Z}_L}$ of which we can think  as inside ${\mathbb T}$. Hence {\bf (USC(${\mathbb Z}$))} holds.
\end{proof}

%
%

\medskip

Now, we prove that rationality of the spectrum is a consequence of the Fuglede conjecture.

\begin{theorem}
Suppose the Fuglede conjecture is true (i.e. both {\bf (S-T(${\mathbb R}$))} and {\bf (T-S(${\mathbb R}$))} holds), then any bounded spectral sets $\Omega$ of $|\Omega|=1$ will have a rational spectrum.
\end{theorem}

\begin{proof}
Let $\Omega$ be a spectral set of $|\Omega|=1$. Let $\Lambda = \Gamma+p{\mathbb Z}$ with $p\in{\mathbb N}$ and $\#\Gamma=p$ be one of its spectra. We can decompose as in (\ref{eq2.1})
$$
\Omega = \bigcup_{|S|=p}\left(A_S+\frac{S}{p}\right), \ \bigcup_{|S|=p}A_S = [0,\frac{1}{p}).
$$
As $\Omega$ is bounded, there are finitely many $S$ such that $|A_S|>0$. By Proposition \ref{prop2.3}, the sets $\frac{1}{p}S$ ($S$ must be one of the $\Omega_x$) have a common spectrum $\Gamma$. Since  {\bf (S-T(${\mathbb R}$))} holds, we must have {\bf (UTC(${\mathbb Z}$))} by Theorem \ref{th0.0}(iii). Taking the finite family ${\mathcal A}$ to be all $S$ above, we have a common tiling set $B'$ for $S$. Hence, $S\oplus B' = {\mathbb Z}$. But any tiling on dimension 1 must be periodic, so there exists $L>0$ such that $S\oplus B = {\mathbb Z}_L$ where $B' = B\oplus L{\mathbb Z}$. Now, {\bf (T-S(${\mathbb R}$))} holds, Theorem \ref{th5.2} implies that {\bf (USC(${\mathbb Z}_n$))} holds. Therefore, all these $S$ must have a common spectrum $\widetilde{\Gamma}$ in ${\mathbb Z}_L$ and $\frac{1}{p}S$ has a common spectrum $\frac{1}{p}\widetilde{\Gamma}$. This spectrum is rational. By Proposition \ref{prop2.3} again, $\Omega$ will have a rational spectrum.
\end{proof}

%
%
%

\medskip

\medskip

\section{Coven-Meyerowitz property}

In this section, we describe how checking the CM-property is sufficient for a proof of the Fuglede conjecture. Note that with the CM-property, the tiling set and the spectrum can be explicitly written.

\begin{definition}\label{def1.1}
Let $A$ be a finite subset of $\bz$ satisfying the CM-property as in Definition \ref{def0.1}.
We call the set
\begin{equation}
\Gamma_{\textup{\L}}:=\left\{\sum_{s\in S_A}\frac{k_s}{s} : s=p^\alpha, k_s\in\{0,\dots,p-1\}\right\}
\label{eqls1}
\end{equation}
the {\it \L aba spectrum of $A$} and

\begin{equation}
\Lambda_{\textup{\L}}:=\Gamma_{\textup{\L}}+\bz
\label{eqls2}
\end{equation}
the {\it \L aba spectrum of $\Omega=A+[0,1)$}. (It was proved in \cite{Lab02} that $\Gamma_{\Laba}$ is a spectrum for $A$ and, with Proposition \ref{prop2.4}, $\Lambda_{\Laba}$ is a spectrum for $\Omega$).

We denote by $M $ the lowest common multiple (lcm) of the numbers in $S_A$. Let $M =p_1^{r_1}\dots p_m^{r_m}$ be the prime factor decomposition of $M$. For each $j=1,\dots,m$ define
$$\alpha_j:=\left\{
              \begin{array}{ll}
                \max\{i: p_j,p_j^2,\dots,p_j^i\in S_A\}, & \hbox{$p_j\in S_A$;} \\
                0, & \hbox{$p_j\not\in S_A$.}
              \end{array}
            \right.
$$
and
\begin{equation}
\frac 1{p_{\textup{\L}}}:=\frac1{p_1^{\alpha_1}\dots p_m^{\alpha_m}},
\label{eqlp1}
\end{equation}
the {\it \L aba period} of the spectrum $\Lambda_{\textup{\L}}$.
\end{definition}

\medskip

\begin{definition}\label{def1.2}
Let $A$ be a finite subset of $\bz$ that satisfies the CM-property. Define the set $B\subset \bz$ by the associated polynomial
\begin{equation}
B(x)=\prod \Phi_s(x^{t(s)})
\label{eqcmt}
\end{equation}
where the product is taken over all the prime powers $s$ which are factors of $M=\lcm(A)$ and $s\not\in S_A$ and $t(s)$ is the largest factor of $M$ which is prime with $s$. We call the set
\begin{equation}
\mathcal T_{CM}:=B\oplus M\bz
\label{eqcmt2}
\end{equation}
the {\it Coven-Meyerowitz tiling set} for $A$. (It was proved in \cite{CoMe99} that $\T_{CM}$ is a tiling set for $A$).
\end{definition}

\medskip

We will now prove Theorem \ref{th0.2}. We will split the proof into several parts. The following proposition does not assume any CM-property, but it tells us that $p_{\Laba}$ plays a special role for the spectral set.

\begin{proposition}\label{lem1.4}
Let $\Lambda$ be a spectrum for $\Omega=A+[0,1)$, $0\in\Lambda$. Then the minimal period $p_{\min}$ of $\Lambda$ is of the form $p_{\min}=\frac1P$ where $P$ is a divisor of $p_\Laba$.
\end{proposition}

\begin{proof}

Since 1 is a period of $\Lambda$ by Proposition \ref{prop2.4}, $1$ is in the period lattice $p_{\min}{\mathbb Z}$ as defined in the introduction and hence $p_{\min} = \frac1P$ for some positive integer $P$.

\medskip

 Let $P=q_1^{\beta_1}\dots q_r^{\beta_r}$ be the prime factor decomposition of $P$. For any $\ell\leq \beta_j$, the number $\frac1{q_j^\ell}$ is an integer multiple of $\frac1P$ and therefore it is a period of $\Lambda$, so $\frac1{q_j^\ell}$ is in $\Lambda$ since $0$ is. Note that the Fourier transform of $\chi_{\Omega}$ is given by
$$
\widehat{\chi_{\Omega}}(\xi)  = A(e^{2\pi i \xi})\widehat{\chi_{[0,1]}}(\xi) = A(e^{2\pi i \xi})e^{\pi i \xi}\frac{\sin(\pi \xi)}{\pi \xi}
$$
and $\frac1{q_j^\ell}$ is a zero for the Fourier transform (since $\Lambda$ is a spectrum), so we must have $A\left(e^{2\pi i/q_j^l}\right)=0$. But then the cyclotomic polynomial $\Phi_{q_j^\ell}$ has to divide $A(x)$ and this implies that $q_j=p_k$ for some $k$ and also, since this holds for all $\ell\leq \beta_j$, we have that $\beta_j\leq \alpha_k$. Thus $P$ is a divisor of $p_\Laba$.
\end{proof}

A corollary is also deduced which sheds some light on the rationality of the spectrum.

\begin{corollary}
Suppose there is no cyclotomic polynomial factor in $A(x)$, then the minimal period any spectrum has to be 1.
\end{corollary}

\begin{proof}
If there is no cyclotomic polynomial factor in $A(x)$, then $p_\Laba =1$ by definition. The conclusion follows directly from Proposition \ref{lem1.4}.
\end{proof}
\medskip

The following lemmas describe the relation of the {\L}aba spectrum and also of the Coven-Meyerowitz tiling set to the {\L}aba period.

\begin{lemma}\label{lem1.1}
Let $p_1,\dots,p_m$ be some distinct prime numbers, $r_1,\dots,r_m$ some positive integers and $k(i,j),l(i,j)\in\{0,\dots,p_j-1\}$ for $i=1\dots m$, $j=1\dots r_j$. If
\begin{equation}
\sum_{j=1}^m\sum_{i=1}^{r_j}\frac{k(i,j)}{p_j^i}\equiv \sum_{j=1}^m\sum_{i=1}^{r_j}\frac{l(i,j)}{p_j^i} \ (\mod \bz),
\label{eq1.1.1}
\end{equation}
then $k(i,j)=l(i,j)$ for all $i,j$.
\end{lemma}

\begin{proof}
We proceed by induction on $S:=\sum_j r_j$. If $S=1$ then we have just one prime power in the denominators and the result is clear. Assume \eqref{eq1.1.1} holds. Let $P:=p_1^{r_1}\dots p_m^{r_m}$. Multiply \eqref{eq1.1.1} by $P$:
$$\sum_j\sum_i\frac{P}{p_j^i}k(i,j)\equiv \sum_j\sum_i \frac{P}{p_j^i}l(i,j).$$
Then the term that corresponds to $\frac{P}{p_1^{r_1}}$ is the only one that is not divisible by $p_1$. We can reduce the equality $\mod p_1$ and we get that $k(1,r_1)=l(1,r_1)$. Cancel these terms and the new $r_1$ becomes $r_1-1$ so we can apply the induction hypothesis to obtain the result.
\end{proof}

\medskip

\begin{lemma}\label{lem1.2}
The \L aba spectrum $\Lambda_{\textup{\L}}$ has minimal period $\frac1{p_{\textup{\L}}}$.
\end{lemma}

\begin{proof}
First, we show that if $p$ is a prime and $p^\alpha$ is not in $S_A$ then $\frac{1}{p^\alpha}$ is not a period for $\Lambda_{\textup{\L}}$. If not then, since $0\in\Lambda_{\textup{\L}}$ we have that $\frac{1}{p^\alpha}$ is in $\Lambda_{\textup{\L}}$. Then there exist some numbers $k_s$ as in \eqref{eqls1} such that
$$\frac1{p^\alpha}\equiv \sum_{s\in S_A}\frac{k_s}{s} \ (\mod\bz).$$
But this would contradict Lemma \ref{lem1.1}.

Next, we show that $\frac{1}{p_j^{\alpha_j}}$ is a period of $\Lambda_{\textup{\L}}$. We show by induction on $i\leq \alpha_j$ that $\frac1{p_j^i}$ is a period for $\Lambda_{\textup{\L}}$. For $i=1$, we have
$$b:=\sum_{s\in S_A}\frac{k_s}{s}+\frac1{p_j}=\frac1{p_j}+\frac{k_{p_j}}{p_j}+\sum_{s\in S_A,s\neq p_j}\frac{k_s}{s}$$
and $a:=\frac{1}{p_j}+\frac{k_{p_j}}{p_j}=\frac{k_{p_j}+1}{p_j}$ if $k_{p_j}<p_j-1$ and $a\equiv 0 \ (\mod \bz)$ if $k_{p_j}=p_j-1$. Thus $b\in \Gamma_{\Lambda_{\textup{\L}}}+\bz$ and $\frac{1}{p_j}$ is a period.

Assume now $\frac{1}{p_j^i}$ is a period, $i<\alpha_j$. Note that by the definition of $\alpha_j$, we have that $p_j,\dots, p_j^{i+1}$ are all in $S_A$. Take
$$b:=\sum_{s\in S_A}\frac{k_s}{s}+\frac1{p_j^{i+1}}=\frac1{p_j^{i+1}}+\frac{k_{p_j^{i+1}}}{p_j^{i+1}}+\sum_{s\in S_A,s\neq p_j^{i+1}}\frac{k_s}{s}$$
and $$a:=\frac{1}{p_j^{i+1}}+\frac{k_{p_j^{i+1}}}{p_j^{i+1}}=\frac{k_{p_j^{i+1}}+1}{p_j^{i+1}}$$ if $k_{p_j^{i+1}}<p_j-1$ and $a= \frac{1}{p_j^i}$ if $k_{p_j^{i+1}}=p_j-1$. Thus, since $\frac{1}{p_j^i}$ is a period, we have $b\in \Gamma_{\textup{\L}}+\bz$ and $\frac{1}{p_j^{i+1}}$ is a period.

The next step is to show that $\frac1{p_\Laba}$ is a period. Take $p_1^{\alpha_1}$, $p_2^{\alpha_2}$. They are mutually prime, so there exist integers $m_1,m_2$ such that $m_1p_1^{\alpha_1}+m_2p_2^{\alpha_2}=1$. Then $m_2\frac{1}{p_1^{\alpha_1}}+m_1\frac1{p_2^{\alpha_2}}=\frac1{p_1^{\alpha_1}p_2^{\alpha_2}}$. Since both $\frac{1}{p_1^{\alpha_1}}$ and $\frac1{p_2^{\alpha_2}}$ are periods, it follows that $\frac{1}{p_1^{\alpha_1}p_2^{\alpha_2}}$ is a period. By induction, we get that $\frac1{p_\Laba}$ is a period for $\Lambda_\Laba$.

Finally, we show that $\frac1{p_\Laba}$ is the minimal period. Let $p$ be the minimal period. Since $1$ is a period we have that $p$ is of the form $p=\frac1k$ with $k\in\bz$. Consider the prime factor decomposition of $k=q_1^{\beta_1}\dots q_t^{\beta_t}$. Then $\frac{1}{q_j^i}$ is a period for any $i\leq \beta_j$. From the statements above, we get that all the $q_j$ must be among the factors $p_k$ of $p_\Laba$, $q_j=p_k$ and the power $\beta_j$ has to be less than or equal to the corresponding $\alpha_k$. Thus $k=p_\Laba$ and the result follows.
\end{proof}

\medskip

\begin{lemma}\label{lem1.3}
The Coven-Meyerowitz tiling set for $A$ is contained in $p_\Laba\bz$.
\end{lemma}

\begin{proof}
Let $s=p_j^k$ be a prime power factor of $M$ which is not in $S_A$. We have that
$$\Phi_s(x^{t(s)})=1+x^{p_j^{k-1}t(s)}+x^{2p_j^{k-1}t(s)}+\dots+x^{(p_j-1)p_j^{k-1}t(s)}.$$
Since $s$ is not in $S_A$ it follows that $k>\alpha_j$. Also $t(s)$ is the largest factor of $M$ which is prime with $s$ so it contains all the other primes $p_l$, $l\neq j$ to the largest power (the one in the factorization of $M$). Thus $p_j^kt(s)$ is divisible by $p_\Laba$, and therefore all the powers that appear in $\Phi_{p_j^k}(x^{t(s)})$ are divisible by $p_\Laba$. So $B$ is contained in $p_\Laba\bz$. Clearly $M$ is divisible by $p_\Laba$ so $\mathcal T_{CM}\subset p_\Laba\bz$.
\end{proof}

\medskip

We now complete the proof of Theorem \ref{th0.2}.

\medskip

\noindent{\it Proof of Theorem \ref{th0.2}.} (i)  \  We check that the ({\bf Strong S-T(${\mathbb Z}$)}) holds if we have CM-property for the spectral sets on ${\mathbb Z}$. Take a finite union of intervals with integer endpoints which is spectral. We can take it to be of the form $\Omega=A+[0,1)$. Since $\Omega$ is spectral it follows that $A$ is a spectral subset of $\bz$. Then $A$ satisfies the CM-property. Take a spectrum $\Lambda$ for $\Omega$. By Lemma \ref{lem1.4}, the minimal period of $\Lambda$ is of the form $\frac1P$ with $P$ a divisor of $p_\Laba$. We have to check that $\Omega$ tiles $\br$ with a subset of $P\bz$. But, we know that $\Omega$ tiles $\br$ with the Coven-Meyerowitz tiling set $\mathcal T_{CM}$ which by Lemma \ref{lem1.3} is contained in $p_\Laba\bz$ and, since $P$ is a divisor of $p_\Laba$, we have also that $p_\Laba\bz\subset P\bz$. Thereofore $\mathcal T_{CM}\subset P\bz$.

\medskip

(ii) \ We prove this part by showing that {\bf (USC(${\mathbb Z}$))} holds. Let ${\mathcal B}$ be a finite family of sets in $\bz$ and let  $A\subset \bz$ be such that $A\oplus B=\bz_L$ for all $B\in\mathcal B$. By translation, we may assume all $B$ are in $\bz^{+}\cup\{0\}$. We claim that all $S_{B}$ are equal for all $B\in {\mathcal B}$. Let ${\mathbb N}_L = \{0,1\cdots,L-1\}$ and then $A\oplus B=\bz_L$ implies
$$
A(x)B(x) = (1+x+\cdots+x^{L-1})Q(x) = \prod_{d|L}\Phi_d(x)Q(x)
$$
where $Q(x)$ is some polynomial with integer coefficients and $Q(1)=1$. We therefore see that $S_A\cup S_B = S_{{\mathbb N}_L}$ and the union is disjoint (otherwise $Q(1)\neq 1$). Hence, $S_B = S_{{\mathbb N}_L}\setminus S_A$ for all $B\in{\mathcal B}$. Now,  by the hypothesis, all $B$ in ${\mathcal B}$  have the CM-property, so they have the {\L}aba spectrum $\Gamma_{{\Laba}}$, note that $\Gamma_{{\Laba}}$ depends only on $S_{B}$ and the sets $S_{B}$ are now all equal. Therefore, the {\L}aba spectrum is the common spectrum of all $B$.
\qquad$\Box$

\medskip

\medskip

\begin{acknowledgements}
This work was partially supported by a grant from the Simons Foundation (\#228539 to Dorin Dutkay).
\end{acknowledgements}

\bibliographystyle{alpha}	
\bibliography{eframes}

\end{document}